\documentclass[12pt,twoside]{article} 
\usepackage{amsmath}
\usepackage{fancyhdr}
\usepackage{amsthm}
\usepackage{amsfonts}
\usepackage{amssymb}
\usepackage{amscd}
\usepackage{latexsym}
\usepackage{amsthm}
\usepackage{graphicx}
\usepackage{graphics}
\usepackage{afterpage}
\usepackage[colorlinks=true, urlcolor=blue,  linkcolor=black, citecolor=black]{hyperref}
\usepackage{color}
\date{} 

\begin{document}

\centerline{\bf Existence and uniqueness of functional
differential equations with n delay } 

\centerline{} 

\centerline{\bf {Bahloul Rachid}} 
\centerline{Faculty of Sciences and Technology, Fez, Morocco} 
\centerline{bahloulr33@hotmail.com}

\centerline{}

\newtheorem{Theorem}{\quad Theorem}[section] 

\newtheorem{Definition}[Theorem]{\quad Definition} 

\newtheorem{Proposition}[Theorem]{\quad Proposition} 

\newtheorem{Corollary}[Theorem]{\quad Corollary} 

\newtheorem{Notation}[Theorem]{\quad Notation} 

\newtheorem{Lemma}[Theorem]{\quad Lemma} 

\newtheorem{Remark}[Theorem]{\quad Remark} 

\newtheorem{Example}[Theorem]{\quad Example} 

\centerline{}

\begin{abstract} In this paper we give a necessary and suffcient conditions for
the existence and uniqueness of periodic solutions of functional differential
equations with n delay $\frac{d}{dt}x(t) = Ax(t) + \sum_{j=1}^{n}Bx(t-r_{j}) + f(t)$. The conditions
are obtained in terms of R-boundedness of operator valued Fourier multipliers.
\end{abstract} 

{\bf Mathematics Subject Classification:} xxxxx \\


\section{Introduction} Let A and B be two closed linear operators defined on a Banach space X
with domains D(A) and D(B), respectively such that $D(A) \subset D(B)$.
In this paper we show existence and uniqueness of solutions for the following
differential equation with n delay

\begin{eqnarray}\label{e}
\left\{
\begin{array}{ccccc}
\displaystyle{\frac{d}{dt}x(t) = Ax(t) + \sum_{j=1}^{n}Bx(t-r_{j}) + f(t)}\\\\
x(0) = x(2\pi).
\end{array}
\right.
\end{eqnarray}
where $f \in L^{p}([-r_{2\pi} ,0], X)$ for some $1 \leq p < \infty, r_{2\pi} = 2\pi N$ ($N \in \mathbb{N}$) and we suppose B is bounded.
The theory of operator-valued Fourier multipliers has attracted the attention of many papers in recent years. For example, this theory was used in \cite{1}
to obtain results about equations $\frac{dx(t)}{dt} = Ax(t) + f(t)$, and in \cite{11} to obtain
results about delay equation $\frac{dx(t)}{dt} = Ax(t) + F(x_{t}) + f(t)$.
In \cite{6}, S.Bu studied $L^{p}$-Maximal Regularity of Degenerate delay Equations with Periodic Conditions.
We note that in the special case when $B = 0$, maximal regularity of Eq. (\ref{e}) has been studied by Arendt and Bu in $L^{p}$-spaces case and Besov spaces case [\cite{1},
\cite{2}], Bu and Kim in TriebelLizorkin spaces case \cite{8}. The corresponding integro-differential equations were treated by Keyantuo and Lizama [\cite{17}, \cite{18}], Bu and
Fang \cite{7}.
In this paper, we characterize the existence and uniqueness for the n delay
equation (\ref{e}) under the condition that X is a UMD space. Here the operator
A is not necessarily the generator of a $C_{0}$-semigroup. We use the operator valued multiplier Fourier method.
The organisation of this work is as follows : In section 2, we present preliminary results on UMD spaces and $L^{P}$-multiplier. In section 3, we study the
existence of periodic strong solution for Eq.(\ref{e}) with finite delay. In section
4, we give the main abstract result ( theorem [4.2] ) of this work.\\
1) for every $f \in L^{p}(\mathbb{T}; X); 1 < p < \infty$, there exists a unique $2\pi$-periodic strong $L^{p}$-solution of Eq. (\ref{e}).\\
2) $(ikI - A - \sum_{j=1}^{n}B_{j,k})$ has bounded invertible for all $k \in \mathbb{Z}$ and $\{ik(ikI - A - \sum_{j=1}^{n}B_{j,k})^{-1}\}_{k \in \mathbb{Z}}$ is R-bounded.

\section{Preliminary Notes} 
Let $X$ be a Banach Space. Firstly, we denote By $\mathbb{T}$ the group defined as the quotient $\mathbb{R}/2 \pi \mathbb{Z}$. There is an identification between functions on $\mathbb{T}$ and $2\pi$-periodic functions on $\mathbb{R}$. We consider the interval $[0, 2\pi$) as a model for $\mathbb{T}$.

\begin{Definition}
A Banach space X is said to be UMD space if the Hilbert transform is bounded
on $L^{p}(\mathbb{R}, X)$ for all $1<p<\infty$.
\end{Definition}

\begin{Example}: \cite{9}\\
1.Any Hilbert space is an UMD space.\\
2. $L^{p} (0.1)$ are UMD spaces for every $1 < p < \infty$.\\
3. Any closed subspace of a UMD space is a UMD space.
\end{Example}

\begin{Definition} \cite{1}\\
A family of operators $T=(T_{j})_{j \in \mathbb{N}^{\ast}}\subset B(X,Y)$ is called $R$-bounded (\textbf{ Rademacher bounded or randomized bounded}), if there is a constant $C > 0$ and
$p \in [1, \infty)$ such that for each $n \in N, T_{j} \in $T$, x_{j}\in X$ and for all independent, symmetric, $\left\{-1,1\right\}$-valued random variables $r_{j}$ on a probability space ($\Omega, M, \mu$) the inequality
$$\left\|\sum_{j=1}^{n} r_{j} T_{j} x_{j}\right\|_{L^{p}(0,1;Y)}\leq C \left\|\sum_{j=1}^{n} r_{j} x_{j}\right\|_{L^{p}(0,1;X)}$$
is valid. The smallest $C$ is called $R$-bounded of  $(T_{j})_{j \in \mathbb{N}^{\ast}}$ and it is denoted by $R_{p}$($T$).
\end{Definition}

\begin{Definition} \cite{11}\\
For $1\leq p < \infty$ , a sequence $\left\{M_{k}\right\}_{k \in \mathbb{Z}} \subset \mathbf{B}(X,Y)$ is said to be an $L^{p}$-multiplier if for each $f \in L^{p}(\mathbb{T}, X)$, there exists $ u\in$ $L^{p}(\mathbb{T}, Y)$ such that $\hat{u}(k) = M_{k}\hat{f}(k)$ for all $k \in \mathbb{Z}$.
\end{Definition}

\begin{Proposition} \cite[Proposition 1.11]{1}
Let $X$ be a Banach space and  $\left\{M_{k}\right\}_{k \in \mathbb{Z}}$ be an $L^{p}$-multiplier, where $1 \leq p < \infty$. Then the set $\left\{M_{k}\right\}_{k \in \mathbb{Z}}$ is $R$-bounded.
\end{Proposition}
\begin{Theorem} \textbf{(Marcinkiewicz operator-valud multiplier Theorem)}.\\
Let $X$, $Y$ be UMD spaces and  $\left\{M_{k}\right\}_{k \in \mathbb{Z}} \subset B(X, Y)$. If the sets $\left\{M_{k}\right\}_{k \in \mathbb{Z}}$ and
$\left\{k(M_{k+1}-M_{k})\right\}_{k \in \mathbb{Z}}$ are \\$R$-bounded, then $\left\{M_{k}\right\}_{k \in \mathbb{Z}}$ is an $L^{p}$-multiplier for 
$ 1 < p < \infty $.
\end{Theorem}
We observe that the condition of R-boundedness for $(M_{k})_{k \in \mathbb{Z}}$ is necessary.

\begin{Remark}\label{r1} \cite{13}\\
Let $f \in L^{1}(\mathbb{T};X)$. If $g(t) = \int_{0}^{t}f(s)ds$ and $k \in \mathbb{Z}, k \neq 0,$ then 
\begin{center}
$\hat{g}(k) = \frac{i}{k}\hat{f}(0) - \frac{i}{k}\hat{f}(k)$
\end{center}
\end{Remark}
\section{A criterion for periodic solutions}
\begin{Notation}
Let $k \in \mathbb{Z}$. Denote by $B_{j, k} := e^{-ikr_{j}}B$, \\$\Delta_{k} = (ikI - A - \sum_{j=1}^{n}B_{j,k})$ and
$\sigma_{\mathbb{Z}}(\Delta):=\{ k \in \mathbb{Z}: \Delta_{k} \text{is not bijective} \}$\\
$H^{1,p}(\mathbb{T}; X) = \left\{ u \in L^{p}(\mathbb{T}, X)   : \exists  v \in  L^{p}(\mathbb{T}, X) , \hat{v}(k) = ik \hat{u}(k) for\   \   all\  \  k \in \mathbb{Z}       \right\}$
\end{Notation}

\begin{Definition}
Let $f \in L^{p}(\mathbb{T}; X)$. A function $x \in H^{1,p}(\mathbb{T}; X)$ is said to be a $2\pi$-periodic strong $L^{p}$-solution of 
Eq. (\ref{e}) if $x(t) \in  D(A)$ for all $t \geq 0$ and Eq. (\ref{e}) holds almost every where.
\end{Definition}

\begin{Lemma} \cite[Lemme  2.1]{1}
Let   $1 \leq p < \infty$ and $u,v \in L^{p}(\mathbb{T}; X)$. Then the following assertions are equivalent:\\
(i) $\displaystyle{ \int_{0}^{2\pi}v(s)ds = 0}$ and there exists $x \in X$ such that  $u(t) = x + \int_{0}^{t}v(s)ds.$\\
(ii) $\hat{v}(k) = ik\hat{u}(k)$   \ for any \ $k \in \mathbb{Z}$.
\end{Lemma}

\begin{Definition} 
For $1 \leq p < \infty$, we say that a sequence $\left\{M_{k}\right\}_{k \in \mathbb{Z}} \subset \mathbf{B}(X, Y)$ is an ($L^{p}, H^{1,p}$)-multiplier, if for each $f \in L^{p}(\mathbb{T}, X)$
there exists $u \in H^{1,p}(\mathbb{T}, Y)$ such that $\hat{u}(k) = M_{k}\hat{f}(k)\ \ \text{ for all}\ \ k \in \mathbb{Z}.$
\end{Definition}

\begin{Lemma} 
Let  $1 \leq p < \infty$ and $(M_{k})_{k \in \mathbb{Z}} \subset  \mathbf{B}(X)\  \ (\mathbf{B}(X)$ is the set of all bounded linear operators from $X$ to $X$). Then the following assertions are equivalent:\\
(i) $(M_{k})_{k \in \mathbb{Z}}$  is an ($L^{p}, H^{1,p}$)-multiplier.\\
(ii) $(ikM_{k})_{k \in \mathbb{Z}}$ is an ($L^{p}, L^{p}$)-multiplier.
\end{Lemma}

\begin{Proposition} Let $A$ be a closed linear operator defined on an UMD space $X$. Suppose that $\sigma_{\mathbb{Z}}(\Delta) = \phi$ .Then the following assertions are equivalent :\\
(i) $\{ik(ikI - A - \sum_{j=1}^{n}B_{j,k})^{-1}\}_{k \in \mathbb{Z}}$ is an $L^{p}$-multiplier for $1 < p < \infty$\\
(ii) $\{ik(ikI - A - \sum_{j=1}^{n}B_{j,k})^{-1}\}_{k \in \mathbb{Z}}$ is $R$-bounded.
\end{Proposition}

\begin{proof}
By [1, Proposition 1.11] it follows that (i) implies (ii). Conversely, define $M_{k} = ik(C_{k} - A)^{-1}$, where $C_{k}=ikI- \sum_{j=1}^{n}B_{j,k}$,
By Theorem 2.6 is sufficient to prove that the set $(k(M_{k+1} - M_{k}))_{k \in \mathbb{Z}}$ is R-bounded. We claim first that the set $(\sum_{j=1}^{n}B_{j,k})_{k \in \mathbb{Z}}$ is R-bounded.\\
since given $x_{j} \in D(A)$ we have :
 \begin{align*}
\left\| \sum_{l=1}^{m} r_{l}(\sum_{j=1}^{n}B_{j,l}) x_{l}\right\|_{L^{p} (0,1; X)}^{p} &= \displaystyle{\int_{0}^{1}}\left\| \sum_{l=1}^{m} r_{l}(t) B(\sum_{j=1}^{n}e^{-ilr_{j}} x_{l})\right\|_{X}^{p}dt\\
&=\displaystyle{\int_{0}^{1}}\left\| B(\sum_{l=1}^{m} r_{l}(t) \sum_{j=1}^{n}e^{-ilr_{j}} x_{l})\right\|_{X}^{p}dt\\
&\leq ||B||^{p}\displaystyle{\int_{0}^{1}}\left\| \sum_{l=1}^{m} r_{l}(t) \sum_{j=1}^{n}e^{-ilr_{j}} x_{l}\right\|_{X}^{p}dt
\end{align*}
By (Lemma 1.7, [1]) we obtain that\\
$\left\| \sum_{l=1}^{m} r_{l}(\sum_{j=1}^{n}B_{j,l}) x_{l}\right\|_{L^{p} (0,1; X)}^{p} \leq 2n^{p}||B||^{p}\displaystyle{\int_{0}^{1}}\left\| \sum_{l=1}^{m} r_{l}(t) x_{l}\right\|_{X}^{p}dt$\\
We conclude that
$$\left\| \sum_{l=1}^{m} r_{l}(\sum_{j=1}^{n}B_{j,l}) x_{l}\right\|_{L^{p} (0,1; X)} \leq 2^{1/p}n||B||.$$
and the claim is proved.
Next. We note the following identities
\begin{align*}
k\left[ M_{k+1} - M_{k}\right]&= k\left[ i(k+1)(C_{k+1} - AD)^{-1} - ik(C_{k}-AD)^{-1}\right]\\
&=k(C_{k+1} - AD)^{-1} [i(k+1)(C_{k} - AD) - ik(C_{k+1} - AD)](C_{k} - AD)^{-1} \\
&=k(C_{k+1} - AD)^{-1} [ik(C_{k} - C_{k+1}) + i(C_{k} - A)](C_{k} - AD)^{-1} \\
&=k(C_{k+1} - AD)^{-1} [ik(C_{k} - C_{k+1})(C_{k} - AD)^{-1}+iI]\\
&=\frac{-ik}{k+1}M_{k+1}(C_{k} - C_{k+1})M_{k} + \frac{k}{k+1}M_{k+1}.
 \end{align*}
We have
$$C_{k} - C_{k+1} =-iI+\sum_{j=1}^{n}Be^{-ikr_{j}}(1-e^{-ir_{j}}).$$
Since products and sums of R-bounded sequences is R-bounded \cite[Remark 2.2]{11}. Then $\{k(M_{k+1} - M_{k})\}_{k \in \mathbb{Z}}$ is R-bounded and  by theorem 2.6, $\{M_{k}\}_{k \in \mathbb{Z}}$ is an $L^{p}$-multiplier.
\end{proof}

\begin{Theorem}\label{31}
Let $X$ be a Banach space. Suppose that for every $f \in L^{p}(\mathbb{T}; X)$ there exists a unique strong solution of Eq. (\ref{e}) for $1 \leq p <  \infty$. Then
\begin{enumerate}
\item for every $k \in \mathbb{Z}$ the operator $\Delta_{k}=( ikI-A-\sum_{j=1}^{n}B_{j,k})$ has bounded inverse
\item  $\left\{ ik\Delta^{-1}_{k}    \right\}_{k \in \mathbb{Z}}$ is $R$-bounded.
\end{enumerate}
\end{Theorem}
Before to give the proof of Theorem (\ref{31}), we need the following Lemma.

\begin{Lemma}\label{33}
if $( ikI-A-\sum_{j=1}^{n}B_{j,k}(x)) = 0$ for all $k \in \mathbb{Z}$, then $u(t) = e^{ikt}x$ is a $2\pi$-periodic strong $L^{p}$-solution of the following equation (\ref{e}) 
corresponing to the function $f = 0$.
\end{Lemma}

\begin{proof}
$( ikI-A-\sum_{j=1}^{n}B_{j,k}(x)) = 0 \Rightarrow ikx = Ax + \sum_{j=1}^{n}B_{j,k}x$.\\ 
We have $u(t) = e^{ikt}x$ then\\
$\displaystyle{u'(t) = ike^{ikt}x = e^{ikt}(ikx)= e^{ikt}[Ax + \sum_{j=1}^{n}B_{j,k}x]=Au(t) + \sum_{j=1}^{n}Bu(t-r_{j}).}$\\
{\bf Proof of Theorem \ref{31}}
1) Let $k \in \mathbb{Z}$ and $y \in X$. Then for $f(t) = e^{ikt}y$ , there exists $x \in H^{1,p}(\mathbb{T}; X)$ such that:
$$\displaystyle\frac{d}{dt}x(t) = Ax(t) + \sum_{j=1}^{n}Bx(t-r_{j})+ f(t)$$
Taking Fourier transform, by Lemma 3.3 we have : 
$$\hat{x'}(k)=ik\hat{x}(k)=A\hat{x}(k)+ \sum_{j=1}^{n}B_{j,k}\hat{x}(k) +\hat{f}(k).$$
Then we obtain : $( ikI-A-\sum_{j=1}^{n}B_{j,k})\hat{x}(k) = \hat{f}(k)=y \Rightarrow ( ikI-A-\sum_{j=1}^{n}B_{j,k})$ is surjective.\\
If $( ikI-A-\sum_{j=1}^{n}B_{j,k})u =0$, then by Lemma \ref{33} $x(t) = e^{ikt}u$ is a 2$\pi$-periodic strong $L^{p}$-solution of Eq. (\ref{e}) corresponing to the function $f = 0$
Hence $x(t) = 0$ and $u = 0$ then $(ikI-A-\sum_{j=1}^{n}B_{j,k})$ is injective.\\
2) Let $f \in L^{p}(\mathbb{T}, X)$. By hypothesis, there exists a unique $x \in H^{1;p}(\mathbb{T}, X)$
such that the Eq. (\ref{e}) is valid. Taking Fourier transforms, we deduce that
$\hat{x}(k) = (ikI-A-\sum_{j=1}^{n}B_{j,k})^{-1}\hat{f}(k)$ for all $k \in \mathbb{Z}$. Hence
$$ik\hat{x}(k) = ik(ikI-A-\sum_{j=1}^{n}B_{j,k})^{-1}\hat{f}(k)$$
On the other hand, since $x \in  H^{1;p}(\mathbb{T}, X)$, there exists $v \in L^{p}(\mathbb{T}, X)$ such that 
$\hat{v}(k)=ik\hat{x}(k) = ik(ikI-A-\sum_{j=1}^{n}B_{j,k})^{-1}\hat{f}(k)$ i.e $\{ ik \Delta^{-1}_{k} \}_{k \in \mathbb{Z}}$ is an
$L^{p}$-multiplier. Then $\{ik\Delta^{-1}_{k} \}_{k \in \mathbb{Z}}$ is R-bounded.
\end{proof}

\section{Existence of mild solutions of Eq. (\ref{e})}
It is well known that in many important applications the operator A is
the infinitesimal generator of $C_{0}$-semigroup $(T(t))_{t \geq 0}$ on the space X.
Let A be a generator of a $C_{0}$-semigroup $(T(t))_{t \geq 0}$.

\begin{Definition}
Assume that A generates a $C_{0}$-semigroup $(T(t))_{t \geq 0}$ on X. A
function $x$ is called a mild solution of Eq. (\ref{e}) if :
$$x(t) = T(t) \varphi + \int_{0}^{t} T(t - s)(\sum_{j=1}^{n}Bx(s-r_{j})+f(s))ds \ \text{for} \ 0 \leq t \leq 2\pi.$$
\end{Definition}

\begin{Remark} \cite[Remark 4.2]{14}\\
Let $(T(t))_{t \geq 0}$ be the $C_{0}$-semigroup generated by $A$. If $g: [0, a] \rightarrow X$ is a continuous function, then $\displaystyle{\int_{0}^{t}\int_{0}^{s}T(t - \xi)g(\xi)d\xi ds} \in D(A)$ and
$$A\int_{0}^{t}\int_{0}^{s}T(t - \xi)g(\xi)d\xi ds = \int_{0}^{t}( T(t - s) - I)g(s)ds \;\;\mbox{for all }\;\;0 \leq t \leq a.$$
\end{Remark}

\begin{Lemma}\label{lem} \cite{10}\\
Assume that A generates a $C_{0}$-semigroup $(T(t))_{t \geq 0}$ on X, if $x$ is a mild
solution of Eq. (\ref{e}) then
$$x(t) = \varphi + A\int_{0}^{t}x(s)ds+ \int_{0}^{t}(\sum_{j=1}^{n}Bx(s-r_{j})+f(s))ds \ \text{for} \ 0 \leq t \leq 2\pi.$$
\end{Lemma}

\begin{Theorem}
Assume that A generates a $C_{0}$-semigroup $(T(t))_{t \geq 0}$ on X and
$f \in L^{p}(\mathbb{T}, X)$ for some $1 \leq p < \infty$, if $x$ is a mild solution of Eq. (\ref{e}). Then
$$(ikI - A -\sum_{j=1}^{n}B_{j, k}) \hat{x}(k) =  \hat{f}(k) \  \text{for all} \  k \in \mathbb{Z}.$$
\end{Theorem}

\begin{proof}
 Let $x$ be a mild solution of Eq. (\ref{e}). Then by Lemma \ref{lem}, we have \\
$$x(t) = \varphi + A\int_{0}^{t}x(s)ds+ \int_{0}^{t}(\sum_{j=1}^{n}Bx(s-r_{j})+f(s))ds$$
For $t = 2\pi$, we have
$$x(2\pi) = \varphi + A\int_{0}^{2\pi}x(s)ds+ \int_{0}^{2\pi}(\sum_{j=1}^{n}Bx(s-r_{j})+f(s))ds; $$
Since: $x(2\pi) =  \varphi$, then
\begin{align*}
&A\int_{0}^{2\pi}x(s)ds+ \int_{0}^{2\pi}(\sum_{j=1}^{n}Bx(s-r_{j})+f(s))ds=0\\
&\Rightarrow \frac{1}{2\pi}A\int_{0}^{2\pi}x(s)ds+ \frac{1}{2\pi}\int_{0}^{2\pi}(\sum_{j=1}^{n}Bx(s-r_{j})+f(s))ds=0\\
&\Rightarrow \frac{1}{2\pi}A\int_{0}^{2\pi}x(s)ds+ \frac{1}{2\pi}\int_{0}^{2\pi}\sum_{j=1}^{n}Bx(s-r_{j})ds+\frac{1}{2\pi}\int_{0}^{2\pi}f(s))ds=0\\
&\Rightarrow \frac{1}{2\pi}A\int_{0}^{2\pi}e^{-i0s}x(s)ds+ \frac{1}{2\pi}\int_{0}^{2\pi}e^{-i0s}\sum_{j=1}^{n}Bx(s-r_{j})ds+\frac{1}{2\pi}\int_{0}^{2\pi}e^{-i0s}f(s))ds=0\\
&\Rightarrow(0-A-\sum_{j=1}^{n}B_{j,0})\hat{x}(0)= \hat{f}(0),
\end{align*}
 which shows that the assertion holds for $k = 0$.\\
Now, define 
$$\displaystyle{v(t) = \int^{t}_{0}x(s)ds}$$ 
and  
$$g(t) = x(t) - \varphi - \int_{0}^{t}(\sum_{j=1}^{n}Bx(s-r_{j})+f(s))ds$$
by Remark \ref{r1} We have:
\begin{align*}
\hat{v}(k)&=\frac{i}{k}\hat{x}(0) - \frac{i}{k}\hat{x}(k) \\
A\hat{v}(k)&=\frac{i}{k}A\hat{x}(0) - \frac{i}{k}A\hat{x}(k)
\end{align*}
and
\begin{align*}
\hat{g}(k)&= \hat{x}(k)-[\frac{i}{k}G_{0}\hat{x}(0) - \frac{i}{k}G_{k}\hat{x}(k)]-[\frac{i}{k}\hat{f}(0) - \frac{i}{k}\hat{f}(k)]\\
&=\hat{x}(k)-\frac{i}{k}G_{0}\hat{x}(0) + \frac{i}{k}G_{k}\hat{x}(k)-\frac{i}{k}\hat{f}(0)+ \frac{i}{k}\hat{f}(k)
\end{align*}
\end{proof}

\begin{Corollary}
Assume that A generates a $C_{0}$-semigroup $(T(t))_{t \geq 0}$ on X
and let $f \in L^{p}(\mathbb{T}, X): 1 \leq p < \infty$ and $x$ be a mild solution of Eq. (\ref{e}). If
$(ikI - A -\sum_{j=1}^{n}B_{j, k})$ has a bounded inverse. Then $(ikI - A -\sum_{j=1}^{n}B_{j, k})$ 
is an $L^{p}$-multiplier.
\end{Corollary}

\begin{proof}
Let $f \in L^{p}(\mathbb{T},X)$ then from Theorem (4.4) we have:\\
$\hat{x}(k) = (ikD_{k}-AD_{k} - G_{k})^{-1}\hat{f}(k)$ for all $f \in L^{p}(\mathbb{T}; X)$, then \\
$(ikI- A -  \sum_{j=1}^{n} B_{j,k})^{-1}$ is an $L^{p}$-multiplier.
\end{proof}

\section{Main Result} Our main result in this work is to establish that the converse of theorem (3.7) and corollary
(4.5) is true, provided X is an UMD space.

\begin{Theorem} (Fejer Theorem) : Let $f \in L^{p}(\mathbb{T}, X)$. Then 
$$f = \lim_{n \rightarrow +\infty} \sigma_{n}(f)$$
where $\sigma_{n}(f)= \frac{1}{n+1} \sum^{n}_{m=0} \sum^{m}_{k=-m} e_{k} \hat{f}(k)$, with $e_{k}(t) = e^{ikt}.$
\end{Theorem}

\begin{Theorem}
Let X be an UMD space and $A : D(A) \subset X \rightarrow X $ be a closed
linear operator. Then the following assertions are equivalent for $1 < p < \infty.$\\
1) for every $f \in L^{p}(\mathbb{T}, X)$ there exists a unique strong $L^{p}$-solution of Eq.(\ref{e}).\\
2) $\sigma_{\mathbb{Z}}(\Delta) = \phi$ and $\{ ik \Delta^{-1}_{k} \}_{k \in \mathbb{Z}}$ is R-bounded.
\end{Theorem}

\begin{proof}
$1\Rightarrow2)$ see Theorem 3.7.\\
$1\Leftarrow2)$  Let $f \in L^{p}(\mathbb{T}; X)$ . Define $\Delta_{k} = (ikI-A - \sum_{j=1}^{n} B_{j,k})$ ,\\
By Proposition 3.6,  the family $\left\{ik \Delta^{-1}_{k}\right\}_{k \in \mathbb{Z}}$ is an $L^{p}$-multiplier it is equivalent to\\
the family $\left\{ \Delta^{-1}_{k}\right\}_{k \in \mathbb{Z}}$ is an $L^{p}$-multiplier that maps $L^{p}(\mathbb{T}; X)$ into $H^{1,p}(\mathbb{T}; X)$,\\
 namely there exists $x \in H^{1,p}(\mathbb{T}, X)$ such that
\begin{equation}\label{11}
\hat{x}(k)=\Delta^{-1}_{k} \hat{f}(k)= (ikI-A - \sum_{j=1}^{n} B_{j,k})^{-1} \hat{f}(k)
\end{equation}
In particular, $x \in L^{p}(\mathbb{T}; X)$ and there exists $v \in L^{p}(\mathbb{T}; X)$ such\\
 that $\hat{v}(k) = ik \hat{x}(k)$
\begin{equation}\label{12}
\widehat{x'}(k):=  \hat{v}(k) = ik \hat{x}(k)
\end{equation}
By Theorem 5.1 we have for $j \in \{1...n\}$
$$x(t - r_{j}) = \lim_{ l \rightarrow +\infty} \frac{1}{l+1} \sum^{l}_{m=0} \sum^{m}_{k=-m} e^{ikt} e^{-ikr_{j}}\hat{x}(k)$$
Then, since B is bounded linear\\
$$\sum_{j=1}^{n} Bx(t-r_{j}) = \lim_{l \rightarrow +\infty} \frac{1}{l+1} \sum^{l}_{m=0} \sum^{m}_{k=- m}e^{ikt} (\sum_{j=1}^{n} B_{j,k} \hat{x}(k))$$
By (\ref{11}) and (\ref{12}) we have: 
$$\widehat{x'}(k) = ik \hat{x}(k) = A \hat{x}(k) + \sum_{j=1}^{n} B_{j,k} \hat{x}(k) + \hat{f}(k), \  \text{for all} \ k \in \mathbb{Z}$$
Then using that $A$ and B  are  closed we conclude that $x(t) \in D(A)$ [[1], Lemma 3.1] and from the uniqueness theorem of Fourier coefficients that  
$$x'(t) = Ax(t) + \sum_{j=1}^{n} Bx(t-r_{j})+f(t).$$
We have $x \in H^{1,p}(\mathbb{T},X)$ then  by lemma 3.3, $x(0) = x(2 \pi)$, then the Eq. (1) has a unique 2$\pi$-periodic strong $L^{p}$-solution. 
\end{proof}

\begin{Theorem} Let $ 1 \leq p < \infty $. Assume that $A$ generates a $C_{0}$-semigroup $(T(t))_{t \geq 0}$ on $X$. If $\sigma_{Z}(\Delta)=\emptyset$ and  $(ikI-A - \sum_{j=1}^{n} B_{j,k})^{-1}$ is an $L^{p}$-multiplier Then there exists a unique mild solution periodic of Eq. (1).
\end{Theorem}

\begin{proof}
 For $f \in L^{p}(\mathbb{T}; X)$ we define
$$f_{l}(t) =  \frac{1}{l+1} \sum^{l}_{m=0} \sum^{m}_{k=-m} e^{ikt} \hat{f}(k)$$
By the Fej\'er Theorem we can assert that $ f_{l} \rightarrow f $  as $ l \rightarrow \infty $ for the norm in $L^{p}(\mathbb{T}; X)$.
We have $ (ikI-A - \sum_{j=1}^{n} B_{j,k})^{-1}$ is an $L^{p}$-multiplier then there exists $x \in L^{p}(\mathbb{T}; X)$ such that
$\hat{x}(k) =  (ikI-A - \sum_{j=1}^{n} B_{j,k})^{-1}\hat{f}(k)$\\
put 
$$x_{l}(t) =  \frac{1}{l+1}\sum^{l}_{m=0}\sum^{m}_{k=-m}e^{ikt}(ikI-A - \sum_{j=1}^{n} B_{j,k})^{-1} \hat{f}(k)$$
Using again the Fej\'er Theorem we obtain that $x_{n}(t) \rightarrow  x(t)$ (as $n \rightarrow \infty)$ and $x_{n}(t)$ is strong $L^{p}$-solution of Eq. (1) and $x_{n}(t)$ verified
\begin{equation}\label{52}
x_{l}(t) = T(t)\varphi_{l}+ \int^{t}_{0}T(t-s)(\sum_{j=1}^{n} Bx_{l}(s-r_{j}) + f_{l}(s))ds
\end{equation}
With  $t = 2\pi$ we obtain 
$$x_{l}(2 \pi) = T(2 \pi)\varphi_{l}+ \int^{2 \pi}_{0}T(2 \pi-s)(\sum_{j=1}^{n} Bx_{l}(s-r_{j}) + f_{l}(s))ds.$$
from which we infer that the sequence $( \varphi_{l} )_{n}$ is convergent to some element $\varphi$ as $l \rightarrow \infty$( $\varphi_{l} = x_{l}(0) =x_{l}(2 \pi)$).
Moreover, $\varphi$ satisfies the condition 
\begin{equation}
\varphi = T(2\pi) \varphi +  \int^{2 \pi}_{0}T(2 \pi-s)(\sum_{j=1}^{n} Bx(s-r_{j}) + f(s))ds.
\end{equation}
Taking the limit as $l$ goes to infinity in (\ref{52}), we can write
$$x(t) = T(t)\varphi+ \int^{t}_{0}T(t-s)(\sum_{j=1}^{n} Bx(s-r_{j}) + f(s))ds:= g(t)$$
$g(2\pi) = T(2\pi)y + \int^{2 \pi}_{0}T(2 \pi-s)(\sum_{j=1}^{n} Bx(s-r_{j}) + f(s))ds  \overbrace{=}^{(5)} \varphi = g(0)$\\
Then $x(2\pi) = \varphi \Rightarrow x(2\pi) = x(0)$, we conclude that $ x $ is a $2\pi$- periodic mild solution of Eq. (1).
\end{proof}

{\bf Acknowledgements.} This is a text of acknowledgements.

{\bf Received: Month xx, 20xx}

\end{document}